\documentclass[11pt]{amsart}
\usepackage{amssymb,amsmath,amsthm,amsfonts,amsopn,url,MnSymbol,color,hyperref}
\usepackage[all]{xy}
\theoremstyle{plain}
\newtheorem{thm}{Theorem}[section]
\newtheorem{prop}[thm]{Proposition}
\newtheorem{lemma}[thm]{Lemma}
\newtheorem{cor}[thm]{Corollary}
\newtheorem{defpr}[thm]{Definition/Proposition}
\theoremstyle{definition}
\newtheorem{defn}[thm]{Definition}
\newtheorem*{defn*}{Definition}
\newtheorem{question*}{Question}
\newtheorem{example}[thm]{Example}
\newtheorem*{example*}{Example}
\newtheorem*{rmk*}{Remark}
\newcommand{\field}[1]{\mathbb{#1}}
\newcommand{\N}{\field{N}}

\newcommand{\ideal}[1]{\mathfrak{#1}}
\newcommand{\m}{\ideal{m}}

\newcommand{\p}{\ideal{p}}

\newcommand{\func}[1]{\mathrm{#1} \,}

\newcommand{\depth}{\func{depth}}

\newcommand{\hgt}{\func{ht}}

\DeclareMathOperator{\Ass}{Ass}

\newcommand{\im}{\func{im}}

\newcommand{\arrow}[1]{\stackrel{#1}{\rightarrow}}

\newcommand{\ra}{\rightarrow}

\DeclareMathOperator{\ann}{ann}
\DeclareMathOperator{\grade}{grade}
\DeclareMathOperator{\Hom}{Hom}
\newcommand{\FF}{\mathcal{F}}

\newcommand{\li}
 {\leftfootline}
\newcommand{\lic}[2]{{#1}^{\li}_{#2}}
\newcommand{\ic}[2]{{#1}^-_{#2}}
\newcommand{\Ri}{{-\rm{Rs}}}
\newcommand{\Ric}[2]{{#1}^\Ri_{#2}}
\newcommand{\EHUi}{{-\rm{EHU}}}
\newcommand{\EHUic}[2]{{#1}^\EHUi_{#2}}
\newcommand{\onto}{\twoheadrightarrow}

\newcommand{\into}{\hookrightarrow}
\newcommand{\tmod}{\tau_{\mathcal{M}}}
\newcommand{\tid}{\tau_{\mathcal{I}}}
\newcommand{\cnd}{\mathfrak{f}}
\newcommand{\cond}[1]{\cnd_{#1}}
\DeclareMathOperator{\Soc}{Soc}
\DeclareMathOperator{\Min}{Min}
\DeclareMathOperator{\Sym}{Sym}
\DeclareMathOperator{\Frac}{Frac}
\DeclareMathOperator{\Supp}{Supp}
\newcommand{\be}{\mathbf{e}}
\newcommand{\vect}[2]{{{#1}_1, \dotsc, {#1}_{#2}}}
\newcommand{\cA}{\mathcal{A}}

\author{Neil Epstein}
\address{Department of Mathematical Sciences \\ George Mason University \\ Fairfax, VA  22030}
\email{nepstei2@gmu.edu}

\author{Bernd Ulrich}
\address{Department of Mathematics\\
Purdue University\\
West Lafayette, Indiana 47907}
\email{ulrich@math.purdue.edu}
\thanks{The second named author was supported in part by a grant from the National Science Foundation.}

\title{Liftable integral closure}

\date{\today}
\begin{document}
\begin{abstract}
We develop the basic properties of an essentially new closure operation on submodules, the \emph{liftable integral closure} of a submodule, including its relationships with the two
prevailing notions of integral closure of submodules.  We show that for a quite general class of local rings, every finite length module may be represented as a quotient of the form $T/L$, where $T$ is torsionless
and integrally dependent on $L$.
\end{abstract}

\maketitle

\section*{Introduction}
The classical notion of integral closure of ideals (developed in the 1940s) became so useful that in 1987, Rees \cite{Rees-redmod-nono} defined integral closure of modules.  Since then, other definitions were proposed, not all of them equivalent, but each having its advantages.  On the other hand, in the words of Huneke and Swanson \cite[in the middle of page 303]{HuSw-book}, ``Every choice of definition has its own problems."

Some definitions (most notably the one in \cite{EHU-Ralg-nono}) place a value on independence of embedding.  That is, if $L \subseteq M \subseteq N$, then the integral closure of $L$ in $M$ should be the intersection of $M$ with the integral closure of $L$ in $N$.  According to this approach also, if $N$, $T$ are ``big enough'' modules containing $L$, then the integral closure of $L$ in $N$ should agree with the integral closure of $L$ in $T$.

In the present work, we take the opposite approach.  That is, if $L \subseteq M$, the `liftable integral closure' of $L$ in $M$, denoted $\lic{L}{M}$, depends \emph{only} on the embedding, in the sense that it only depends on the module structure of the quotient module $M/L$.  As it happens, this approach has multiple surprising applications, not least a theorem about representing finite length modules in terms of integral extensions of torsionless modules.

We wish to stress here that we are not claiming this notion to be the ``right'' definition of the integral closure of a submodule.  We agree, of course, with the Huneke-Swanson quote above.  Instead, the operation explored here should be seen as a different and useful perspective in the theory of integral closure of modules.

The structure of the paper is as follows:  In \S\ref{sec:history}, we discuss some pre-existing notions of integral closure and their interrelationships, paying special attention to the case of not necessarily finitely generated modules.  In \S\ref{sec:properties}, we define the \emph{liftable integral closure} $\lic{L}{M}$ relative to a submodule inclusion $L \subseteq M$ and develop its basic properties (see especially Lemma~\ref{lem:fin} and Proposition~\ref{pr:list}).  In \S\ref{sec:test} we give notions of `test ideals' relative to liftable integral closure, and we show that there are usually only one or two such ideals per ring (denoted $\tid$ and $\tmod$), despite a plethora of definitions (see Definition/Proposition~\ref{pr:test}).  In \S\ref{sec:char}, we characterize these ideals based largely on the dimension of the ring.  For 0-dimensional rings, $\tid$ is the socle (Proposition~\ref{pr:dim0test}); for 1-dimensional Cohen-Macaulay rings, $\tid$ is the conductor (Theorem~\ref{thm:dim1test}), thus generalizing a theorem of J. Vassilev; for equidimensional excellent unmixed rings of dimension at least two, $\tmod=\tid=0$ (Theorem~\ref{thm:hightest}).  Finally, in \S\ref{sec:torsionless}, we provide some surprising consequences of this latest result.  For instance (Theorem~\ref{thm:torless}), given a finite length module $M$ over an excellent equidimensional unmixed local ring of dimension at least two, there exists a finite free module $F$ and submodules $L \subseteq T \subseteq F$ such that $T$ is integral over $L$ in $F$.  This raises the question: what other sorts of finitely generated modules may be represented as a quotient of finite torsionless modules sharing a common integral closure?

\section{Some history of integral closures of modules}\label{sec:history}

Relative to an embedding, there is a commonly held notion of the integral closure of a submodule of a free module, especially in the case where the submodule has a rank.  More generally, we make the following definition:

\begin{defn}[Integral closure in a free module]
Let $L \subseteq F$ be $R$-modules, where $F$ is free.  Let $S$ be the (naturally graded) symmetric algebra over $R$ defined by $F$, and let $T$ be the subring of $S$ induced by the inclusion $L \subseteq F$.  Note that $S$ is $\N$-graded and generated in degree 1 over $R$, and $T$ is an $\N$-graded subring of $S$, also generated over $R$ in degree 1.  Then the \emph{integral closure of $L$ in $F$}, denoted $L^-_F$, is the degree 1 part of the integral closure of the subring $T$ of $S$.
\end{defn}

This definition agrees with all the usual definitions of the integral closure relative to an embedding in any situation where $R$ is Noetherian and $F$ is finitely generated and free.  See \cite[Chapter 16]{HuSw-book} for details.  As for integral closure of a submodule of an arbitrary (finitely generated) module, there seem to be two prevailing notions -- one by Rees \cite{Rees-redmod-nono}, and one by Eisenbud, Huneke, and the second author of this work \cite{EHU-Ralg-nono}.  We give them both below:

\begin{defn}[Rees integral closure]
Let $L \subseteq M$ be modules\footnote{Rees only defined this notion for finitely generated modules, and he required that the valuation rings $V$ be Noetherian.  However, one may employ the usual tricks to show that for finitely generated modules, the definition given here is equivalent to his.} over a Noetherian ring $R$.  First suppose $R$ is a domain, and let $Q$ be its fraction field.  Then for $x\in M$, we say that $x$ is in the \emph{Rees integral closure of $L$ in $M$}, written $x\in \Ric{L}{M}$, if for every valuation ring $V$ between $R$ and $Q$, $x$ is in the image $LV$ of $L \otimes_R V$ in $M \otimes_R Q$ (or equivalently, in $M \otimes_R V$).

In the general case, we write $x\in \Ric{L}{M}$ if for every minimal prime $\p$ of $R$, we have $x + \p M \in \Ric{\left(\frac{L + \p M}{\p M}\right)}{M/\p M}$, as modules over the ring $R/\p$.
\end{defn}

\begin{defn}[EHU integral closure]
Let $L \subseteq M$ be modules over a Noetherian ring $R$.  Then for $x\in M$, we say that $x$ is in the \emph{EHU integral closure of $L$ in $M$}, written $x\in \EHUic{L}{M}$, if for every $R$-homomorphism of the form $g: M \ra F$ for free modules $F$, we have $g(x) \in g(L)^-_F$.
\end{defn}

These would seem at first glance to have nothing to do with each other.  But we have the following:

\begin{prop}[Valuative Criterion]\label{pr:valcrit}
Let $L \subseteq M$ be modules over a Noetherian ring $R$ and let $x\in M$.  Then $x\in \EHUic{L}{M}$ if and only if for every map $M \ra F$ to a free module and every homomorphism $R \ra V$ to a valuation ring $V$ with kernel a minimal prime of $R$, we have $(L + Rx)' = L'$, where $(-)'$ denotes tensoring with $V$ and taking images in $V \otimes_R F$.
\end{prop}

Under the assumption that $M$ is finitely generated (and using only Noetherian valuation rings), this was proved as \cite[the ``\emph{Valuative Criterion}'' part of Theorem 2.2]{EHU-Ralg-nono}.  We postpone the proof of the general case until page~\pageref{pf:valcrit}, after we have developed some tools for dealing with non-finitely generated modules.

\begin{cor} Let $R$ be a Noetherian ring and $L \subseteq M$ $R$-modules.  Then $\Ric{L}{M} \subseteq \EHUic{L}{M}$, with equality if $M$ is free (in which case both coincide with the first definition of this paper) or if $R$ is a domain. \end{cor}

In general, however, Rees integral closure may be strictly smaller than EHU integral closure:

\begin{example}\cite[immediately after the proof of Theorem 2.2]{EHU-Ralg-nono}
Let $k$ be a field, $R:=k[x]/(x^2)$, and $M := Rx$.  Then $\EHUic{0}{M}=M$, but $\Ric{0}{M} = 0$.
\end{example}

\section{Liftable integral closure - basic properties}\label{sec:properties}

Next we come to the primary definition of the current paper:

\begin{defn}[Liftable integral closure]
Let $L \subseteq M$ be $R$-modules.  Let $\pi: F \onto M$ be a surjection of a free $R$-module $F$ onto $M$.  Let $K := \pi^{-1}(L)$.  Then the \emph{liftable integral closure} of $L$ in $M$ is \[
\lic{L}{M} := \pi(K^-_F).
\]
\end{defn}

\begin{rmk*}
As noted earlier, the same approach is taken when defining tight closure of a submodule.  That is, tight closure is defined for submodules of free module, and then one extends to tight closure in arbitrary modules by lifting to a free module surjecting onto the ambient module and taking the tight closure there \cite[8.1--8.2]{HHmain-nono}.  One can even use symmetric algebras to do so (provided the free module is finite and there are enough test elements), as developed in McCulloch's thesis \cite[Chapter 3]{Mcc-thesis}.  In joint work in progress, the first named author and M. Hochster show how to use this approach toward a wide range of closure operations on ideals to get closure operations on submodules, recovering tight closure (in McCulloch's situation) and liftable integral closure as special cases \cite{nmeHo-heq}.

What we call liftable integral closure probably coincides with Brenner's submersive version of integral closure (\emph{op. cit.}) at least in the case of finitely generated modules.
\end{rmk*}

Given such a surjection, $\lic LM$ is obviously a submodule of $M$.  Next, we show that this is well-defined (i.e. independent of $\pi$).

\begin{prop}\label{pr:indep}
Let $L \subseteq M$ be a submodule inclusion.  Then $\lic{L}{M}$ is independent of choice of the free module surjection onto $M$.
\end{prop}

\begin{proof}
We may assume first that $L=0$.

Let $p: F \onto M$ and $\pi: G \onto M$ be surjective maps, where $F$, $G$ are free $R$-modules.  Let $U := \ker p$ and $V := \ker \pi$.  Since $F$, $G$ are projective, there exist liftings $\gamma: F \ra G$ and $\delta: G \ra F$ of $p$, $\pi$ respectively.  That is, $\pi \circ \gamma = p$ and $p \circ \gamma = \pi$.

Now, let $A := \Sym F$ and $B := \Sym G$.  Let $g: A \ra B$ be the unique lifting of $\gamma$ to the symmetric algebra level.  Let $\langle V \rangle$ denote the ideal of $B$ generated by the set $V \subseteq B_1 = G$, and $\langle U \rangle$ the corresponding ideal of $A$.  Then by composing with the natural map $B \ra B/\langle V \rangle$, we have a standard graded $R$-algebra map $\phi: A \ra B/\langle V \rangle$.  Moreover, since $U = \ker (F \ra G/V)$, which is the degree 1 part of the map $\phi$, it follows that $\langle U \rangle \subseteq \ker \phi$.  Hence, there is an induced map $\bar{\gamma}: A / \langle U \rangle \ra B / \langle V \rangle$.

Now, let $x\in M$, let $y \in F$ such that $p(y) = x$ and $y \in U^-_F$, and let $z \in G$ such that $\pi(z) = x$.  We want to show that $z\in V^-_G$.  To say that $y \in U^-_F$, it is equivalent to say that the image $\bar y$ of $y$ is nilpotent in $\Sym(F) / \langle U \rangle$.  That is, $\bar y^t =\bar 0$ for some $t\in \N$.  But then since $\bar \gamma$ is a ring homomorphism, it follows that $\overline{ \gamma(y)}^t = \bar \gamma(\bar y)^t =  \bar \gamma(\bar y^t) = \bar \gamma(\bar 0) = \bar 0$ in $\Sym(G) / \langle V \rangle$.  That is, $\gamma(y)^t \in \langle V \rangle$, whence $\gamma(y) \in V^-_G$.

On the other hand, $z - \gamma(y) \in V$.  To see this, we have $\pi(z-\gamma(y)) = \pi(z) - \pi(\gamma(y)) = \pi(z)-p(y) = x-x=0$, so $z-\gamma(y) \in \ker \pi = V$.  Putting these two facts together, we have \[
z = (z-\gamma(y)) + \gamma(y) \in V + V^-_G = V^-_G,
\]
as was to be shown.
\end{proof}

Note that if $M$ is free, then $\lic{L}{M} = L^-_M$.  In particular, for any ideal $I$, $\lic{I}{R} = \bar{I}$.

The following lemma allows us frequently to reduce to the finitely generated case.

\begin{lemma}[Reduction lemma]
\label{lem:fin}
Let $\Lambda$ be any index set.  Suppose $L \subseteq M$ are modules, and we have submodules $\{L_\lambda\}_{\lambda \in \Lambda}$, $\{M_\lambda\}_{\lambda \in \Lambda}$ of $M$ such that $M = \sum_{\lambda\in\Lambda} M_\lambda$, $L = \sum_{\lambda\in \Lambda} L_\lambda$, and $L_\lambda \subseteq L \cap M_\lambda$ for each $\lambda\in \Lambda$. Then \[
\lic{L}{M} = \bigcup_{\textrm{finite subsets } \sigma \subseteq \Lambda}
\lic{\left(\sum_{\lambda\in\sigma} L_\lambda\right)}{ \sum_{\lambda\in\sigma} M_\lambda}
\]
\end{lemma}

\begin{proof}
First, by definition of liftable integral closure, we may immediately quotient out by $L$ and assume that every $L_\lambda=0$.

Next, for each $\lambda\in \Lambda$, let $\pi_\lambda: F_\lambda \onto M_\lambda$ be a surjection from a free $R$-module.  For each $\lambda$, let $X_\lambda$ be a set of free module generators of $F_\lambda$ over $R$, defined in such a way that all the sets $\{X_\lambda \mid \lambda \in \Lambda \}$ are disjoint from one another.  Then with $F := \bigoplus_{\lambda \in \Lambda} F_\lambda$, we have a short exact sequence \[
0 \ra U \ra F \arrow{\pi} M \ra 0, 
\]
where $\pi$ is induced from the maps $\pi_\lambda$.

Now let $z\in \lic{0}{M}$.  Let $y\in F$ such that $\pi(y) = z$.  Then in $S := \Sym(F)$, there is an equation of integrality of the form \begin{equation}\label{eq:inty}
y^n + a_1 y^{n-1} + \cdots + a_{n-1} y + a_n = 0,
\end{equation}
where each $a_j \in (US)^j$.  That is each $a_j$ is an $R$-linear combination of $j$-fold formal products of elements of $U$, where \emph{each} such element is an element of a direct sum of only finitely many of the $F_\lambda$s.  For the same reason, $y$ itself can be expressed in terms of only finitely many of the $F_\lambda$s.  Putting this together, it follows that there is some finite list $\lambda_1, \dotsc, \lambda_t$ of elements of $\Lambda$ such that Equation~\ref{eq:inty} is an equation of integrality in $F' := \bigoplus_{i=1}^t F_{\lambda_i}$ of $y$ over the submodule $U' := U \cap F'$.

Combine this with the fact that $\sum_{i=1}^t M_{\lambda_i} \cong F'/U'$ canonically, and the statement follows.
\end{proof}

At this point, we possess the tools to prove the Valuative Criterion in general:

\begin{proof}[Proof of Proposition~\ref{pr:valcrit}]\label{pf:valcrit}
Throughout, we may assume that $M$ is a free $R$-module (albeit for different reasons from when we are dealing with liftable integral closure).

For the first inclusion, take any $x\in \EHUic{L}{M} = L^-_M$.  Then by Lemma~\ref{lem:fin}, there is some finitely generated free submodule $G$ of $M$ such that, letting $K := L \cap G$, we have $x \in K^-_G$.  Then by the finitely generated case of the Valuative Criterion, $x\in \Ric{K}{G} \subseteq \Ric LM$.

For the opposite inclusion, we first prove the result when $R$ is a domain.  Let $Q$ be its field of fractions.  Let $T$ be the subring of $\Sym_R(M)$ generated by $L$.  Let $W$ be a valuation domain between $T$ and $\Frac(T)$.  Then $V := W \cap Q$ is a $Q$-valuation domain containing $R$.  So for any $x$ satisfying the Valuative Criterion, we have $x \in LV$, the image of the map $L \otimes_RV \ra F \otimes_RV$.  But $LV \subseteq W$, so $x\in W$.  Since $W$ was arbitrary, it follows that $x$ is integral over $T$, and hence $x \in \EHUic{L}{M}$.

Finally we show the general case, where $R$ is not necessarily a domain but $x$ satisfies the Valuative Criterion in $M$ over $L$.  Then for each $\p \in \Min(R)$, $\bar{x} \in M/\p M$ satisfies the valuative criterion over $(L+\p M)/\p M$ as $(R/\p)$-modules.  But $R/\p$ is a domain, so by the previous paragraph, $x\in (L + \p M)^-_{M/\p M}$ as $(R/\p)$-modules.  Now fix a free basis of $M$ over $R$, and let $\{M_\lambda\}$ be the set of all free submodules of $M$ generated by finite subsets of that basis.  Then by Lemma~\ref{lem:fin} (and since there are only finitely many minimal primes of $R$), there is some $M' := M_\lambda$ such that if we let $L' := L \cap M'$, we have $x\in M'$ and for each $\p \in \Min(R)$, $x \in (L' + \p M')^-_{M'/\p M'}$ as $(R/\p)$-modules.  But then by the usual relationship between integral closure in finitely generated free modules and minimal primes, it follows that $x \in (L')^-_{M'} \subseteq L^-_M = \EHUic{L}{M}$, as was to be shown.
\end{proof}

We next collect several important properties of liftable integral closure:

\begin{prop}\label{pr:list}
Let $R$ be Noetherian and let $L \subseteq M$ be $R$-modules.
\begin{enumerate}
 \item\label{it:idem} {\emph{(Idempotence)}} $\lic{L}{M} = \lic{(\lic LM)}M$.
 \item\label{it:functor} {\emph{(Functoriality)}} If $h: M \ra N$ is a homomorphism of $R$-modules, then $h(\lic{L}{M}) \subseteq \lic{h(L)}{N}$.
 \item\label{it:submods} {\emph{(Submodules)}} If $K \subseteq L$ is a submodule, then $\lic{K}{L} \subseteq \lic{K}{M} \subseteq \lic{L}{M}$.
  \item\label{it:directsum} {\emph{(Direct sums)}} If $N \subseteq T$ is another inclusion of $R$-modules, then
\[
\lic{(L \oplus N)}{M \oplus T} = \lic{L}{M} \oplus \lic{N}{T}
\]
 \item\label{it:tc} {\emph{(Comparison to tight closure)}} Suppose that $M$ is finitely generated and either (a) $R$ has characteristic $p$ and is essentially of finite type over an excellent local ring, or (b) $R$ is finitely generated over a field of characteristic $0$.  Then $L^*_M \subseteq \lic{L}{M}$.
 \item\label{it:Ric} {\emph{(Comparison to Rees integral closure)}} We have $\lic{L}{M} \subseteq \Ric{L}{M}$, with equality if $M$ is free.
 \item\label{it:colon} {\emph{(Colons)}} If $L=\lic{L}{M}$, $J$ is an ideal, and $U \subseteq M$ another submodule, then $\lic{(L :_M J)}{M} = L:_MJ$ and $\overline{(L :_RU)} = L :_RU$.
 \item\label{it:persistence} {\emph{(Persistence)}} Let $R \ra S$ be a homomorphism of Noetherian rings.  Then
\[
 \im (\lic{L}{M} \otimes_RS \ra M \otimes_RS) \subseteq \lic{\left(\im (L\otimes_RS \ra M \otimes_R S)\right)}{M \otimes_RS},
\]
where the closure on the right-hand side is taken as $S$-submodules.
 \item\label{it:flat} {\emph{(Normal base change)}} If in the situation of Property~(\ref{it:persistence}), the ring homomorphism is \emph{normal} (\emph{i.e.} flat with geometrically normal fibers), then the displayed $S$-module containment is an equality.
 \item\label{it:semiprime} {\emph{(Semiprime property)}} For any ideal $J$, $\lic{(JL)}{M} = \lic{(\bar{J}\lic{L}{M})}{M}.$
 \item\label{it:Nak} {\emph{(Nakayama property)}} If $(R,\m)$ is local and Noetherian, and $M/L$ is finitely generated, then $\lic{L}{M} \subseteq L + \m M$.  Thus, $\lic{L}{M}=M$ if and only if $L=M$.
\end{enumerate}
\end{prop}

\begin{proof}
(\ref{it:idem}): We may assume that $M$ is free.  Then by looking at the symmetric algebra of the free module $M$, the property follows from the standard fact that whenever $A \subseteq B$ is an extension of commutative rings, the integral closure of $A$ in $B$ is integrally closed in $B$.

(\ref{it:functor}): First, we may assume $L=0$, so that we want to show that $h(\lic{0}{M}) \subseteq \lic{0}{N}$.

Next, lift to a map $\tilde{h}: F \ra G$ of free modules, so that we get the following commutative diagram with exact rows:
\[\xymatrix{
0 \ar[r] &U \ar[r] \ar[d]_{h_0} &F \ar[r]^{\pi} \ar[d]_{\tilde{h}} &M \ar[r] \ar[d]_h &0 \\
0 \ar[r] &V \ar[r] &G \ar[r]_p &N \ar[r] &0
}
\]
Then $\tilde{h}$ lifts to a map $H: \Sym(F) \ra \Sym(G)$ of symmetric algebras.  Let $x\in \lic{0}{M}$.  Then there is some $y\in F$ with $\pi(y) = x$ such that $y \in U^-_F$, which means that $y$ is in the integral closure of the subalgebra $S_F(U)$ generated by $U$ in $\Sym(F)$.  By persistence of integral closure of subalgebras along ring maps, it follows that $\tilde{h}(y) = H(y)$ is integral over the subalgebra $S_G(V)$ generated by $V$ in $\Sym(G)$, which means that $\tilde{h}(y) \in V^-_G$, whence $h(x) = h(\pi(y)) = p(\tilde{h}(y)) \in \lic{0}{N}$, as required.

(\ref{it:submods}): The first inclusion follows from (\ref{it:functor}) applied to the inclusion map $L \into M$.  The second inclusion is clear since we may assume $M$ is a free module.

(\ref{it:directsum}): This follows from properties~(\ref{it:functor}) and (\ref{it:submods}), applied to the canonical projection and inclusion maps of the direct sum.

(\ref{it:tc}): First, since both tight closure and liftable integral closure are independent of surjection from free modules, we may assume $M$ is free.

Recall that for a regular Noetherian local ring $S$, any submodule of any $S$-module is tightly closed.

Now let $x \in L^*_M$ and let $R \ra V$ be a map to a Noetherian valuation ring with kernel a minimal prime of $R$.  Then $x' \in (LV)^*_{M \otimes_R V} = LV$ (since $V$ is regular and tight closure persists from $R$ to $V$ \cite[Theorem 6.24]{HHbase-nono}), so $x \in \Ric{L}{M} = \lic{L}{M}$.

(\ref{it:Ric}): Let $\pi: F \onto M$ be a surjection from a free module, let $K := \pi^{-1}(L)$, and let $x \in \lic{L}{M}$.  Let $y\in F$ with $\pi(y)=x$.  Going mod a minimal prime, we may assume $R$ is a domain, with quotient field $Q$. Let $V$ be a valuation ring between $R$ and $Q$.  Then by the Valuative Criterion, $y' \in KV$.  But then $x'=\pi(y)' = \pi'(y') \in LV$, so that $x \in \Ric{L}{M}$.

If, on the other hand, $M=F$ is free, then $\Ric{L}{M} = \EHUic{L}{M} = \lic{L}{M} = \ic{L}{M}$.

(\ref{it:colon}): Let $\pi: F \onto M$ be a surjection from a free module.  Let $K := \pi^{-1}(L)$ and $V := \pi^{-1}(U)$.  It follows easily that $K :_F J = \pi^{-1}(L :_M J)$ and $L :_R U = K :_R V$, so it suffices to prove the analogous formulas for $K$, $V$, and $F$.

First let $y \in (K :_F J)^-_F$.  Then in $S := \Sym(F)$, we have an equation of the form: \[
y^n + \alpha_1 y^{n-1} + \cdots + \alpha_n = 0,
\]
where each $\alpha_i \in ((K :_FJ)S)^i$.  Take any $j\in J$.  Then $j^i \alpha_i \in (KS)^i$, so multiplying the displayed equation by $j^n$, we get \[
(jy)^n + (j \alpha_1) (jy)^{n-1} + \cdots + (j^{n-1} \alpha_{n-1}) (jy) + j^n \alpha_n = 0.
\]
This shows that $jy \in K^-_F = K$, as required.

Similarly, if $r\in \overline{(K :_R V)}$, then we have the following equation in $R$: \[
r^n + c_1 r^{n-1} + \cdots + c_n = 0,
\]
where each $c_j \in (K :_R V)^j$.  Now let $v \in V$, and in $S=\Sym(F)$ it follows that $c_j v^j \in (KS)^j$.  So multiplying the displayed equation by $v^n$ in $S$, we get: \[
(rv)^n + (c_1 v) (rv)^{n-1} + \cdots + c_n v^n = 0,
\]
whence $rv \in K^-_F = K$.  Since $v\in V$ was arbitrary, it follows that $r \in  K :_R V$, as required.

(\ref{it:persistence}): We may assume immediately that $M$ is free.  Let $\{M_\lambda\}$ be the set of finitely generated free $R$-submodules of $M$, and let $L_\lambda := L \cap M_\lambda$ for each $\lambda$. Note that $\{M_\lambda\}$ form a direct limit system under inclusion, whose direct limit is $M$, and similarly for the $\{L_\lambda\}$.  Accordingly, let $z \in \im (\lic{L}{M} \otimes_RS \ra M \otimes_RS)$. Then there is some $\lambda$ such that $z \in \im (\lic{(L_\lambda)}{M_\lambda} \otimes_RS \to M \otimes_R S)$.  Hence,
\[
z = \sum_{i=1}^n z_i \otimes s_i,
\]
where $z_i \in \lic{(L_\lambda)}{M_\lambda}$ and $s_i \in S$ for $1 \leq i \leq n$. But since Property~(\ref{it:persistence}) holds for submodules of finitely generated free modules (which follows, in turn, from the persistence of integral closure of subrings), we have that
\[
z_i \otimes s_i \in \im [\lic{(\im L_\lambda \otimes S \to M_\lambda \otimes S)}{M_\lambda \otimes S} \to M \otimes S].
\]
But then by Properties~(\ref{it:functor}) and (\ref{it:submods}), it follows that
\[
 z_i \otimes s_i \in \lic{\left(\im (L\otimes_RS \ra M \otimes_R S)\right)}{M \otimes_RS},
\]
as was to be shown.

(\ref{it:flat}): As usual, we may assume $M$ is free.  We need only show that $\lic{(L \otimes_R S)}{M \otimes_R S} \subseteq \lic{L}{M} \otimes_R S$.  So let $z \in \lic{(L \otimes_R S)}{M \otimes_R S}$.  Let $\{M_\lambda\}$ and  $\{L_\lambda\}$ be as above.  By flatness, it follows that the systems $\{L_\lambda \otimes_R S\}$ and $\{M_\lambda \otimes_R S\}$, of $S$-submodules of $L \otimes S$ and $M \otimes S$ respectively, are also injective direct limit systems, with unions equal to $L \otimes S$ and $M \otimes S$ respectively.  Accordingly, by Lemma~\ref{lem:fin}, there is some $\lambda$ such that $z \in \lic{(L_\lambda \otimes_R S)}{M_\lambda \otimes_R S}$.  But since Property~(\ref{it:flat}) holds for submodules of finitely generated free modules (as the argument in the proof of \cite[Corollary 19.5.2]{HuSw-book} works just as well in this case), it follows that $z\in \lic{(L_\lambda)}{M_\lambda} \otimes_R S$.  By Property~(\ref{it:submods}) on the $R$-module inclusions involved, $\lic{(L_\lambda)}{M_\lambda} \subseteq \lic{L}{M}$, so that by flatness,
\[
z \in \lic{(L_\lambda)}{M_\lambda}\otimes_R S \subseteq \lic{L}{M} \otimes_R S,
\]
as was to be shown.

(\ref{it:semiprime}): 
As usual we may assume $M$ is free.  Use Lemma~\ref{lem:fin} to reduce to the case where $M$ is finitely generated.    Then the statement is clear via use of the Valuative Criterion.

(\ref{it:Nak}): Suppose $L \neq M$.  Let $G \arrow{\psi} F \ra M/L \ra 0$ be a minimal free presentation of $M/L$, and let $U := \im \psi$.  By minimality, $U \subseteq \m F$.  But $\m F$ is integrally closed in $F$, as is clear by tensoring everything with $R/\m$ and noting that subspaces of vector spaces are always integrally closed.  Hence $U^-_F \subseteq (\m F)^-_F = \m F$, so that
\[
\frac{\lic{L}{M}}{L} = \lic{0}{M/L} = \pi(U^-_F) \subseteq \m \pi(F) = \m (M/L) = \frac{L + \m M}{L}.
\]
whence $\lic{L}{M} \subseteq L + \m M$.  The last statement follows from the Nakayama lemma.
\end{proof}

\begin{rmk*}
Property (\ref{it:Nak}) of the above Proposition shows that for finitely generated modules, it makes no sense to say that a module is `liftably integral' over a submodule (though it does for infinite modules; see Proposition~\ref{pr:ihull} below).  This shows that liftable integral closure is quite a distinct notion from previous notions of integral closure of submodules.  For example, let $R$ be any Noetherian local domain that contains an ideal $J$ that is not integrally closed.  Let $I =\bar{J}$ be its integral closure.  Then $\Ric{J}{I} = \EHUic{J}{I}=I$ (since $JV = IV$ for all ring homomorphisms $R \ra V$ with kernel a minimal prime of $R$ and $V$ a rank 1 DVR), but $\lic{J}{I}  \neq I$.
\end{rmk*}

Finally for this section, here is an equivalent way to define liftable integral closure:

\begin{lemma}
Let $R$ be a Noetherian ring, and let $L \subseteq M$ be (\emph{resp.} finite) $R$-modules.

Consider \emph{all} surjections $\pi: Z \onto M$ from $R$-modules (\emph{resp.} from finite $R$-modules) $Z$, and introduce the notation $L^\pi := \pi^{-1}(L)$ and $M^\pi := Z$.  Then \[
\lic{L}{M} = \bigcap_{\pi} \pi(\EHUic{(L^{\pi})}{M^\pi}) = \bigcap_{\pi} \pi( \Ric{({L^\pi})}{M^\pi}).
\]
\end{lemma}

\begin{proof}
Since for any $R$-modules $A \subseteq B$ and any $R$-module map $g: B \ra C$, we have $g(\Ric{A}{B}) \subseteq \Ric{g(A)}{C}$ (since it holds after reducing to the domain case) and $g(\EHUic{A}{B}) \subseteq \EHUic{g(A)}{C}$ (as this operation is defined in terms of maps between modules), and since every (finite) $R$-module is the surjective image of a (finite) free $R$-module, we need only consider surjections from (finite) free modules.  Then the result follows by definition of liftable integral closure and Proposition~\ref{pr:indep}.
\end{proof}

That is, an element $z \in M$ is \emph{liftably integral} over $L$ in $M$ if for all `liftings' $\pi$ of $M$, a preimage of $z$ is in the (Rees or EHU) integral closure of $L^\pi$ in $M^\pi$, and the \emph{liftable integral closure} of $L$ in $M$ is the set of all such elements.

\section{`Test ideals' for liftable integral closure}\label{sec:test}

A main point of this work is to determine what kills liftable integral closures of zero in some `universal' way, inspired by the tight closure notion of ``test elements".

\begin{defpr}\label{pr:test}
Let $(R,\m,k,E)$ be a local Noetherian ring.  Then \begin{enumerate}
\item
\begin{align*}
\tmod &:= \bigcap_{M \text{ artinian}} \ann \lic{0}{M} = \bigcap_{M \text{ fin.\ length}} \ann \lic{0}{M} \\
&= \bigcap_{M \text{ fin.\ gen.}} \ann \lic{0}{M} = \ann \lic{0}{E},
\end{align*}
where the intersections in question are taken over all $R$-modules of the specified types.

\item \[
\tid :=\bigcap_{\text{ideals }I} (I:\bar{I}) =  \bigcap_{\m\text{-primary ideals }I} (I : \bar{I}).
\]

\item $\displaystyle\tid \subseteq  \bigcap_{\m\text{-primary param.\ ideals }J} (J:\bar{J})$, with equality if $k$ is infinite.

\item $\tmod \subseteq \tid$ with equality if $\dim R=0$ or $R$ is approximately Gorenstein.

Whenever $\tmod=\tid$, we call it the \emph{integral test ideal} of $R$.
\end{enumerate}
\end{defpr}

Recall \cite{Ho-purity-nono} that a local ring $(R, \m, k,E)$ is \emph{approximately Gorenstein} if there exists a sequence $\{I_t\}_{t \in \N}$ of irreducible $\m$-primary ideals such that for all $N \in \N$, there exists $t\in \N$ such that $I_t \subseteq \m^N$.  Moreover, we have
\begin{thm}\cite{Ho-purity-nono}
Let $(R,\m,k,E)$ be a local Noetherian ring.  $R$ is approximately Gorenstein if and only if $\hat{R}$ is.  If $R$ is approximately Gorenstein, then there is a sequence $\{I_t\}_{t \in \N}$ of irreducible $\m$-primary ideals and a sequence $i_t: R/I_t \into R/I_{t+1}$ of \emph{injective} $R$-linear maps such that $E \cong \displaystyle \lim_{\rightarrow} (R/I_t)$ under the action of these maps.

Suppose that one of the following conditions holds:
\begin{enumerate}
 \item $\hat{R}$ is reduced.
 \item $R$ is excellent and reduced.
 \item $\depth R \geq 2$.
 \item $R$ is a normal domain.
\end{enumerate}
Then $R$ is approximately Gorenstein.
\end{thm}

Thus, for rings of dimension $\geq 2$, the condition is quite general.

\begin{proof}[Proof of Proposition~\ref{pr:test}]
We will start with the $\tmod$ equalities (that is, part (1)).  For the purposes of the proof, let $\displaystyle \tau_1 := \bigcap_{M \text{ artinian}} \ann \lic{0}{M}$, $\displaystyle \tau_2 := \bigcap_{M \text{ fin.\ length}} \ann \lic{0}{M}$, $\displaystyle \tau_3 := \bigcap_{M \text{ fin.\ gen.}} \ann \lic{0}{M}$, and $\tau_4 := \ann \lic{0}{E}$.

$\tau_1 \subseteq \tau_4$: since $E$ is artinian.

$\tau_4 \subseteq \tau_2$: Let $M$ be a finite length $R$-module, and $G := E_R(M)$, the injective hull of $M$.  Then $G$ is isomorphic to a direct sum of finitely many copies of $E$, say $G := E^n$.  Then $\lic{0}{M} \subseteq \lic{0}{G} \cong (\lic{0}{E})^{\oplus n}$ (by Proposition~\ref{pr:list} (\ref{it:directsum})), and hence any element of $R$ that annihilates $\lic{0}{E}$ must annihilate $\lic{0}{M}$ as well.

$\tau_2 \subseteq \tau_1$: Let $N$ be an artinian $R$-module and $z\in \lic{0}{N}$.  As $N$ is the directed union of its finite length submodules, it follows from Lemma~\ref{lem:fin} that $z\in \lic{0}{M}$ for some finite length submodule $M \subseteq N$.  Since any $c\in \tau_2$ annihilates $\lic{0}{M}$, we have $c z = 0$.

At this point we have shown that $\tau_1 = \tau_2 = \tau_4$.  To complete the proof of the equalities for $\tmod$, we need only show that $\tau_2=\tau_3$.

$\tau_3 \subseteq \tau_2$: since every finite length module is finitely generated.

$\tau_2 \subseteq \tau_3$: Let $N$ be a finitely generated $R$-module and let $c\in \tau_2$.  Let $z\in \lic{0}{N}$.  Since $\lic{}{}$ respects inclusions (by Proposition~\ref{pr:list} (\ref{it:submods})), we have $z\in \lic{(\m^t N)}{N}$ for all integers $t\geq 1$.  Let $N_t := N / \m^t N$.  Since $N_t$ has finite length, it follows that $c$ annihilates $\lic{0}{N_t}$, so that $c \cdot \bar{z} = \bar{0}$ in $N_t$.  In other words, $cz \in \m^t N$ for all $t$, \emph{i.e.} $cz \in \bigcap_{t\geq 1} \m^t N = 0$ by the Krull intersection theorem.  

We proceed with the assertions about $\tid$ (that is, parts (2) and (3)): Let $\displaystyle \tau_5 := \bigcap_{\text{ideals }I} (I:\bar{I})$, $\displaystyle \tau_6 := \bigcap_{\m-\text{primary ideals }I} (I:\bar{I})$, and $\displaystyle \tau_7 := \bigcap_{\m-\text{prim. param. ideals }J} (J:\bar{J})$.  The proof that $\tau_5=\tau_6$ is essentially identical to the proof that $\tau_2=\tau_3$.  $\tau_6 \subseteq \tau_7$ for obvious reasons, so it remains only to see that $\tau_7 \subseteq \tau_6$ when $k$ is infinite.  But it is well known \cite{NR-nono} that in this case, every $\m$-primary ideal $I$ has a minimal reduction $J$ which is generated by a system of parameters.  So for any such $I$, any $z\in \bar{I}$, and any $c\in \tau_7$, since $\bar{I}=\bar{J}$, we have $cz \in J \subseteq I$, so $c \in (I : \bar{I})$.

As for (4): First, it is clear that $\tmod \subseteq \tid$ in all cases, since for any ideal $I$ (with closure being taken as $R$-modules), $\lic{0}{R/I} = \bar{I}/I$, so that $\ann \lic{0}{R/I} = \ann (\bar{I}/I) = (I :_R \bar{I})$.  For the other direction, if $R$ is approximately Gorenstein, then $E$ is a directed union of finite length \emph{cyclic} $R$-modules, so Lemma~\ref{lem:fin} shows that $\tid \subseteq \tmod$ in this case.  Finally, the case $\dim R = 0$ is given in Proposition~\ref{pr:dim0test} below.
\end{proof}

It is also interesting to consider top local cohomology modules. As one may expect, the results are at their cleanest when the ring is Cohen-Macaulay:

\begin{prop}
Let $(R,\m)$ be a Noetherian local ring.  Let $d:= \dim R$.  Then
\[
\tid \subseteq \ann \lic{0}{H^d_\m(R)},
\]
and equality holds whenever $R$ is Cohen-Macaulay with infinite residue field.
\end{prop}

\begin{proof}
As is well-known, whenever $x_1, \dotsc, x_d$ is a system of parameters, we can express the module $H^d_\m(R)$ as a direct limit, namely
\[
 H^d_\m(R) = \lim_{\ra} \frac{R}{(x_1^t, \dotsc, x_d^t)}
\]
where the maps are given by multiplication by the element $x = \prod_{i=1}^d x_i$. To be explicit, let $j_t: R/(x_1^t, \dotsc, x_d^t) \ra H^d_\m(R)$ and $\alpha_{t,n}: R/(x_1^t, \dotsc, x_d^t) \ra R/(x_1^{t+n}, \dotsc, x_d^{t+n})$ be the corresponding maps, and let $\pi_t: R \ra R/(x_1^t, \dotsc, x_d^t)$ be the natural projection map.

To prove the first inclusion, let $a\in \tid$ and $0\neq u\in \lic{0}{H^d_\m(R)}$.  For each $t$, let $\kappa_t := \ker (j_t \circ \pi_t)$.  By Lemma~\ref{lem:fin}, there is some $t$ such that $u \in \lic{0}{j_t(R/(x_1^t, \dotsc, x_d^t))}$.  That is, there is some $y\in R$ such that $y + (x_1^t, \dotsc, x_d^t)$ represents $u$ in the direct limit system and $y \in \overline{\kappa_t}$.  But for some $n$, we have $x^n \kappa_t \subseteq (x_1^{n+t}, \dotsc, x_d^{n+t})$, so that $x^n y \in \overline{(x_1^{n+t}, \dotsc, x_d^{n+t})}$ by functoriality of integral closure of ideals.  Since $a$ annihilates the integral closure of ideals, this implies that $a x^n y \in (x_1^{n+t}, \dotsc, x_d^{n+t})$. But $x^n y + (x_1^{n+t}, \dotsc, x_d^{n+t})$ represents $u$ in the direct limit system, so it follows that $au=0$.  Thus, $a\in \ann \lic{0}{H^d_\m(R)}$.

As for the reverse inclusion (assuming $R$ is Cohen-Macaulay with infinite residue field), choose $a\in \ann \lic{0}{H^d_\m(R)}$ and let $J$ be an $\m$-primary parameter ideal. Since $R$ is Cohen-Macaulay, we have an inclusion $R/J \into H^d_\m(R)$, which in turn induces an inclusion $\bar{J}/J \into \lic{0}{H^d_\m(R)}$, since as $R$-modules we have $\lic{0}{R/J} = \bar{J}/J$. But then since $a$ annihilates $\lic{0}{H^d_\m(R)}$, it annihilates the submodule $\bar{J}/J$, so that $a \in (J :_R \bar{J})$.  Since this holds for all $\m$-primary parameter ideals $J$, the result follows from Proposition~\ref{pr:test}.
\end{proof}

\section{Characterizations of $\tid$ and $\tmod$}\label{sec:char}

Next, we characterize $\tid$ (and sometimes $\tmod$) under certain conditions on the local ring $R$.

\subsection*{Dimension zero} We begin with the following observation.

\begin{prop}\label{pr:dim0lic}
Let $(R,\m)$ be a Noetherian local ring.  The following are equivalent:
\begin{enumerate}
 \item $\dim R=0$.
 \item For every finitely generated $R$-module $M$, $\lic{0}{M} = \m M$.
 \item For every $R$-module inclusion $L \subseteq M$ such that $M/L$ is finitely generated, $\lic{L}{M} = L + \m M$.
\end{enumerate}
\end{prop}

\begin{proof}
It is clear that (2) and (3) are equivalent.

To see that (2) $\implies$ (1), suppose $\dim R \geq 1$.  Then letting $M=R$, we see that
\[
 \lic{0}{M} = 0^-_R = \sqrt{0} \neq \m = \m M.
\]

To see that (1) $\implies$ (2), let $\dim R=0$ and let $M$ be a finite $R$-module.  Let $\pi: F \onto M$ be a surjection from a finite free module.  Say $F = \oplus_{i=1}^t R \be_i$, where the $\be_i$ are free module generators.  Then
\[
 \m F = \bigoplus_{i=1}^t \m \be_i = \bigoplus_{i=1}^t (0^-_R) \be_i \subseteq 0^-_F \subseteq \m F, 
\]
so that all inequalities become equalities and $\m F = 0^-_F$. Then by Proposition~\ref{pr:list} (\ref{it:Nak}), 
\[
 \m M = \pi(\m F) = \pi(0^-_F) \subseteq \lic{0}{M} \subseteq \m M
\]
which finishes the proof.
\end{proof}

\begin{prop}\label{pr:dim0test}
Let $(R,\m,k,E)$ be a Noetherian local ring of dimension $0$.  Then $\tid = \tmod = \Soc R$.
\end{prop}

\begin{proof}
We have $E=\omega_R$, the \emph{canonical module}, which is finitely generated and faithful.  Thus, by Proposition~\ref{pr:dim0lic}, $\tmod = \ann \lic{0}{E} = \ann\, (\m E)$.  But since $R$ is complete, $\ann\, (\m E) =\ann \m = \Soc R$, so $\tmod =\Soc R$.  On the other hand, for any proper ideal $I$, $(I : \bar{I}) = (I : \m) \supseteq (0:\m) = \Soc R$, and $(0 : \bar{0}) = \ann \m = \Soc R$, so $\tid = \Soc R$.
\end{proof}

\subsection*{Dimension one}
To deal the dimension 1 case, we define the \emph{conductor} $\cond{R}$ of a Noetherian ring $R$ to be the ideal $\cond{R} = (R :_R \bar{R})$, where $\bar{R}$ is the integral closure of $R$ in its total ring of fractions.  This agrees with the usual definition (e.g. in \cite[Chapter 12]{HuSw-book}) when $R$ is reduced.

Next, we note the following:

\begin{prop}
Let $(R,\m)$ be a Noetherian local ring.  Then\footnote{Recall in particular the convention that $\grade R = \infty$.} $\grade \cond{R} \geq 1$ if and only if $\bar{R}$ is finitely generated as an $R$-module.
\end{prop}

\begin{proof}
Let $Q$ be the total quotient ring of $R$.

First suppose $\bar{R}$ is finitely generated as an $R$-module.  Say $\bar{R} = \sum_{i=1}^t R z_i$, with each $z_i \in Q$.  We have $z_i = a_i / x_i$, where $a_i \in R$ and $x_i$ is a non-zerodivisor of $R$.  Let $x := \prod_i x_i$.  Then $x$ is a non-zerodivisor of $R$, and $x z_i \in R$ for each $i$, whence $x \bar{R} \subseteq R$, so that $x \in \cond{R}$, giving that ideal positive grade.

Conversely, suppose $\grade \cond{R}\geq 1$.  Then $\cond{R}$ contains a non-zerodivisor, say $x$.  Then $x \bar{R}$ is an $R$-submodule of $R$, hence an ideal, hence finitely generated.  Say $x \bar{R} = (a_1, \cdots, a_t)R$.  Then
\[
 \bar{R} = \sum_{i=1}^t R \cdot \frac{a_i}{x}
\]
is finitely generated as an $R$-module.
\end{proof}

\begin{thm}\label{thm:dim1test}
Let $(R,\m,k,E)$ be a Noetherian local ring of dimension one, with infinite residue field.  Then $R$ is Cohen-Macaulay if and only if $\tid = \cond{R}$; otherwise $\cond{R}=R$ and $\tid$ is proper.
\end{thm}

\begin{proof}
First suppose $R$ is not Cohen-Macaulay.  Then since $\dim R=1$, we have $\depth R=0$, whence $\m\in \Ass R$.  Thus, the total ring of fractions $Q$ of $R$ is $R$ itself, whence $\bar{R} = R$.  Thus, $\cond{R} = R :_R \bar{R} = R :_R R = R$.  On the other hand, $R$ is non-reduced, so that in particular $0 \neq \sqrt{0} = (0)^-$, whence $((0) : (0)^-) \neq R$, so that $\tid = \bigcap_I (I : \bar{I}) \subseteq (0 : (0)^-) \neq R$.

Conversely, suppose $R$ is Cohen-Macaulay. By Proposition~\ref{pr:test}, we have
\begin{align*}
 \tid &= \bigcap_I (I : I^-) = \bigcap_{x \text{ parameter}} ((x) : (x)^-) \\
&= \bigcap_{x \text{ nzd}} ((x) : (x)^-) = \bigcap_{x \text{ nzd}} ((x) : (x \bar{R} \cap R)) \\
&\supseteq \bigcap_{x \text{ nzd}} (xR :_R x \bar{R}) = \bigcap_{x \text{ nzd}} (R :_R \bar{R}) = \cond{R}.
\end{align*}

Now we need only show that $\tid \subseteq \cond{R}$. Let $Q$ be the total quotient ring of $R$. There is some index set $J$ and some set of elements $\{z_j\}_{j\in J}$ of $\bar{R}$ with $\bar{R} = \sum_{j\in J} R z_j$ as an $R$-submodule of $Q$.  We have $\cond{R} = (R :_R \sum_j R z_j) = \bigcap_j (R :_R R z_j)$.  Moreover, each $z_j = a_j/x_j$ for some $a_j \in R$ and some non-zerodivisor $x_j \in R$.  Then $x_j \in (R :_R R z_j)$, so that $x_j z_j \in \bar{R} x_j \cap R = (x_j)^-$.  This means that $((x_j) :_R (x_j)^-) \subseteq ((x_j R) :_R x_j z_j) = (R :_R R z_j)$.  Putting it all together, we get
\[
\tid = \bigcap_I (I :_R I^-) \subseteq \bigcap_j ((x_j) :_R (x_j)^-) \subseteq \bigcap_j (R :_R R z_j) = \cond{R}.
\]
\end{proof}

\begin{rmk*}
A part of Theorem~\ref{thm:dim1test} was proved by Janet C. Vassilev in her doctoral thesis \cite[Theorem 3.11 and Remark 3.12]{Va-thesis} (a special case of which is mentioned in \cite[Example 3.5]{Hu-tcparam}).  Namely, she proved that if $R$ is a one-dimensional local integral domain of characteristic $p>0$ with infinite residue field, the `test ideal' $\tau(R)$ (for tight closure) equals the conductor. However, in any equicharacteristic Noetherian ring of dimension 1, tight closure is identical to integral closure, and hence $\tau(R) = \tid(R)$.  So in our terms, Vassilev showed that when $R$ is a one-dimensional integral domain of prime characteristic, $\tid = \cond{R}$, a result recoverable from our theorem since one-dimensional integral domains are Cohen-Macaulay.
\end{rmk*}

\subsection*{Higher dimension}

\begin{thm}\label{thm:hightest}
Let $(R,\m)$ be either excellent or the homomorphic image of a Gorenstein local ring.  Suppose that $\dim R \geq 2$ and that $R$ is equidimensional with no embedded primes.  Then $\tmod=\tid=0$.
\end{thm}

\begin{proof}
Under the given hypotheses, it is shown in \cite[Propositions 2.7 and 3.8]{HH-canon} that there is a ring $S$, into which $R$ embeds as a subring, such that $S$ is module-finite over $R$ and satisfies Serre's condition (S$_2$), both as an $R$-module and as a ring in its own right.

Accordingly, let $x, y \in \m$ form a regular sequence on $S$ as an $R$-module.  Fix a positive integer $n$, and let $I = (x^{2n}, y^{2n})$.  Note that $x^n y^n \in I^-$, as $(x^n y^n)^2 = (x^{2n})(y^{2n})$ gives an equation of integrality.  Also note that the fact that $x^n, y^n$ form a regular sequence on $S$ makes it easy to show that $IS :_S x^n y^n  = (x^n, y^n)S$.  So we have \begin{align*}
\tid &\subseteq (I :_R I^-) \subseteq (I :_R x^n y^n) \\
&\subseteq (IS :_S x^n y^n) = (x^n, y^n) S \subseteq \m^n S.
\end{align*}
However, $n$ was arbitrary and $S$ is a finitely generated $R$-module.  So by the Krull intersection theorem, $\tmod \subseteq \tid \subseteq \bigcap_n \m^n S = 0.$
\end{proof}

\section{Some surprising consequences}\label{sec:torsionless}

\begin{prop}\label{pr:ihull}
Let $(R,\m,k,E)$ be an excellent, Noetherian, equidimensional local ring of dimension at least $2$ that has no embedded primes.  Then $\lic 0E = E$.
\end{prop}

\begin{proof}
First assume $R=\hat{R}$.

By Theorem~\ref{thm:hightest}, $\lic{0}{E}$ is a faithful $R$-module.  However, any faithful $R$-submodule $L$ of $E$ must equal $E$.\footnote{This particular claim can fail when $R$ is not complete. Indeed, if $J$ is any ideal of $\hat{R}$ that contracts to $0$ in $R$, let $M := (0 :_E J)$.  Then $\ann_{\hat{R}}M = J$ (so that $M \neq E$), but $\ann_R M = R \cap (\ann_{\hat{R}}M) = R \cap J = 0$, so that $M$ is faithful as an $R$-module. To see that such a $J$ can exist, let $R$ be any local domain which is not analytically irreducible, and let $J$ be any minimal prime of $\hat{R}$.  To be even more concrete, set $R := k[x,y]_{(x,y)} / (x^2-y^2-y^3)$, where $k$ is a field of characteristic $\neq 2$ and $x$, $y$ are indeterminates over $k$, and let $J := (x - y\sqrt{1+y})$.}  To see this, consider the short exact sequence \[
0 \ra L \ra E \ra E/L \ra 0.
\]
Taking the Matlis dual, we get the short exact sequence \[
0 \leftarrow R/J \leftarrow R \leftarrow J \leftarrow 0,
\]
where $J = (E/L)^\vee$ and $R/J = L^\vee$.  But then
\[
L = (R/J)^\vee = \Hom_R(R/J,E) \cong (0 :_E J),
\]
 so that $J L = J \cdot (0 :_E J) = 0$.  So since $L$ is faithful, $J=0$, whence $J^\vee = E/L=0$, so that $E=L$.

Finally, we treat the general case, where $R$ is not necessarily complete. First, note that since $R$ is excellent, $\hat{R}$ is equidimensional with no embedded primes.  Next, note that any $R$-submodule $L$ of $E$ is in fact an $\hat{R}$-submodule, and that $L = \hat{R} \otimes_R L$.  In particular, $\lic{0}{E} = \hat{R} \otimes_R \lic{0}{E}$ (over $R$) $= \lic{0}{\hat{R} \otimes_R E}$ (over $\hat R$, by Proposition~\ref{pr:list} (\ref{it:flat}) applied to the ring homomorphism $R \ra \hat R$) $= \lic{0}{E}$ (over $\hat{R}$) $=E$ (by the first part of the proof).  
\end{proof}

The above result is surprising because this could never happen in a finitely generated module.  We find this sufficiently interesting to merit a concrete example:

\begin{example*}
Let $R := k[\![x,y]\!]$, $k$ a field, and $E$ the injective hull of the residue field.  Let $F := \oplus_{i\geq 1} R t_i$, where the $t_i$ are free generators.  For each $i\geq 1$, let $a_i := x^i t_i$, $b_i := y^i t_i$, $c_i := (xy) t_{i+1} - t_i$, and $d_i := (xy)^i t_{2i} - t_i$.  Let $U$ be the submodule of $F$ generated by the $a_i$s, $b_i$s, and $c_i$s.  Note that $E \cong F/U$ (indeed this is essentially the ``inverse powers'' presentation of Macaulay) and one checks readily that every $d_i \in U$.  For each $i\geq 1$, the following equation holds in the symmetric algebra of $F$ over $R$: \[
t_i^2 + 2 d_i t_i + (d_i^2 - a_{2i} b_{2i}) = 0.
\]
This shows that $t_i$ is integral over $U$ for every $i$.  Hence $U^-_F = F$, so that by definition $\lic{0}{E} = E$.
\end{example*}

The general fact yields the following consequence: 

\begin{thm}\label{thm:torless}
Let $R$ be an excellent, Noetherian, equidimensional local ring of dimension at least $2$ that has no embedded primes.  Let $M$ be an Artinian $R$-module.  Then there exist torsionless\footnote{Recall that a module is \emph{torsionless} if it is a submodule of a free module.} $R$-modules $L \subseteq T$ such that $M \cong T/L$ and $T$ is integral over $L$.  Moreover, if $M$ is finitely generated (and hence, has finite length), $T$ and $L$ may also be chosen to be finitely generated.
\end{thm}

\begin{proof}
We do the general case first.  Since $M$ is Artinian, there is some positive integer $n$ such that $j: M \hookrightarrow E^{\oplus n}$, where $E$ is the injective hull of the residue field of $R$.  We consider this injection to be an inclusion of modules.  Let $\pi: F \onto E^{\oplus n}$ be a surjection from a free module, $U := \pi^{-1}(M)$, and $K := \ker \pi$.  Then since $\lic{K}{F} = F$ (by Proposition~\ref{pr:ihull}), we have that $U$ is also integral over $K$, and $U/K \cong M$.

If $M$ is finitely generated then we can pick elements $\vect u t \in U$ that generate $U$ modulo $K$.  Let $U' := \sum_{i=1}^t R u_i$.  Fix a basis of $F$ as a free module over $R$, and let $\FF$ denote the set of finitely generated free submodules $G$ of $F$ generated by parts of this basis such that $U' \subseteq G$.  Note that $\FF$ is a direct limit system under inclusion, whose directed union equals $F$.  Also, $M \subseteq \pi(G)$ for any $G \in \FF$, and $U' \subseteq F = \lic{K}{F}$.  Then by Lemma~\ref{lem:fin}, there is some $G \in \FF$ such that $U' \subseteq \lic{(K \cap G)}{G}$.  Let $L := K \cap G$ and $T := L + U'$.  Then $T$ is integral over $L$ in the finite free module $G$, and $T/L \cong M$.
\end{proof}

Here is a global version:

\begin{thm}
Let $R$ be a Noetherian ring.  Let $M$ be a finite $R$-module, and let $\cA$ be the set of minimal primes $\p$ of $M$ such that $\hgt \p \geq 2$ and $R_\p$ is excellent and equidimensional with no embedded primes.  Assuming $\cA \neq \emptyset$, there exist finite torsionless $R$-modules $L \subseteq T$ such that $M \cong T/L$ and such that for all $\p \in \cA$, $T_\p$ is integral over $L_\p$ (as $R_\p$-modules).
\end{thm}

\begin{proof}
We replace the $E^{\oplus n}$ in the proof of Theorem~\ref{thm:torless} with the injective hull of $M$, so that we get a (canonical) injective map $j: M \hookrightarrow E := E_R(M)$.  But since $R$ is Noetherian and $M$ finitely generated, we have $E = \bigoplus_{\p \in \Supp M} E_R(R/\p)^{\oplus \mu(\p,M)}$, where $\mu(\p,M)$ is the 0th Bass number of $M$ with respect to $\p$, and is always a nonnegative integer.  Also, this works well with localization, so that for any minimal prime $\p$ of $M$, we have $E(M)_\p = E_{R_\p}(\kappa(\p))^{\oplus \mu(\p,M)}$, where $\kappa(\p) = R_\p / \p R_\p$ is the residue field of $\p$.  Thus, whenever we localize at any $\p \in \cA$, we are in the situation of Theorem~\ref{thm:torless}.  Now pick a free module surjection $\pi: F \onto E$, $U := \pi^{-1}(M)$, and $K := \ker \pi$.

Now pick elements $\vect u t \in U$ that generate $U$ modulo $K$, and let $U' := \sum_{i=1}^t R u_i$.  Let $\FF$ be as in the proof of Theorem~\ref{thm:torless}.  For each $\p\in \cA$, $(U+K)_\p$ is integral over $K_\p$ in $F_\p$.  Thus, for \emph{each} such $\p$, there is some $G\in \FF$ such that $U'_\p \subseteq \lic{((K \cap G)_\p)}{G_\p}$ as $R_\p$-modules.  Let $H$ be the sum of all such $G$'s.  By construction, $H$ is also finitely generated and free, and we have $U'_\p \subseteq \lic{((K \cap H)_\p)}{H_\p}$ as $R_\p$-modules, for \emph{all} $\p \in \cA$.  Now let $L := K \cap H$ and $T := L + U'$, and the inclusion $L \subseteq T$ (as submodules of $H$) satisfies the conclusion of the theorem.
\end{proof}

\section*{Acknowledgment}
The authors are grateful to Balakrishnan R for finding an error in an earlier draft of this work.

\providecommand{\bysame}{\leavevmode\hbox to3em{\hrulefill}\thinspace}
\providecommand{\MR}{\relax\ifhmode\unskip\space\fi MR }
\providecommand{\MRhref}[2]{%
  \href{http://www.ams.org/mathscinet-getitem?mr=#1}{#2}
}
\providecommand{\href}[2]{#2}

\end{document}